\documentclass[12pt, reqno, a4paper]{amsart}
\usepackage{amsmath,mathrsfs}
\usepackage{amssymb}
\usepackage{color}
\usepackage{calligra}
\usepackage{dsfont}
\usepackage{pstricks,pst-node}
\usepackage[latin1]{inputenc}
\usepackage{marginnote}

\newcommand\NoBlackBoxes{\global\overfullrule0pt}
\NoBlackBoxes

\newcommand{\N}{\mathbb{N}}

\newcommand{\artanh}{\mathrm{artanh}}

\makeatletter
\let\serieslogo@\relax
\let\@setcopyright\relax

\newtheorem{definition}{Definition}[section]
\newtheorem{theorem}[definition]{Theorem}
\newtheorem{lemma}[definition]{Lemma}

\newenvironment{remark}[1][Remark]{\begin{trivlist}
\item[\hskip \labelsep {\bfseries #1}]}{\end{trivlist}}

\renewcommand{\P}{{\mathbb{P}}}
\newcommand{\E}{{\mathbb{E}}}

\newcommand{\R}{{\mathbb{R}}}

\newcommand{\V}{\mathbb{V}}

\renewcommand{\epsilon}{\varepsilon}
\renewcommand{\phi}{\varphi}

\numberwithin{equation}{section}

\begin{document}

\setcounter{page}{1}

\title[Fluctuations for block spin Ising models]{Fluctuations for block spin Ising models}

\author[Matthias L\"owe]{Matthias L\"owe}
\address[Matthias L\"owe]{Fachbereich Mathematik und Informatik,
Universit\"at M\"unster,
Orl\'{e}ans-Ring 10,
48149 M\"unster,
Germany}

\email[Matthias L\"owe]{maloewe@math.uni-muenster.de}

\author[Kristina Schubert]{Kristina Schubert}
\address[Kristina Schubert]{Fachbereich Mathematik und Informatik,
Universit\"at M\"unster,
Orl\'{e}ans-Ring 10,
48149 M\"unster,
Germany}

\email[Kristina Schubert]{kristina.schubert@uni-muenster.de}


\newcommand{\wlim}{\mathop{\hbox{\rm w-lim}}}
\newcommand{\na}{{\mathbb N}}
\newcommand{\re}{{\mathbb R}}

\newcommand{\vep}{\varepsilon}

\begin{abstract}
We analyze the high temperature fluctuations of the magnetization of the so-called Ising block model. This model was recently introduced by Berthet, Rigollet and Srivastava in \cite{BRS17_blockmodel}. We prove a Central Limit Theorems (CLT) for the magnetization in the high temperature regime. At the same time we show that this CLT breaks down at a line of critical temperatures. At this line we show the validity of a non-standard Central Limit Theorems for the magnetization.
\end{abstract}

\maketitle

\section{Introduction}
In a recent paper Berthet, Rigollet and Srivastava studied a block version of the Curie-Weiss-Ising model \cite{BRS17_blockmodel}.
This model is inspired by extensive studies of block models in the recent past, see e.g.~\cite{AL18}, \cite{Bre15}, \cite{BMS13}, \cite{GMZZ17}, \cite{MNS16}.
On the other hand, similar models were considered a bit earlier in the statistical mechanics literature, see
\cite{collet_CW},
\cite{CGM},
\cite{Fedele_U},
\cite{gallo_barra_contucci},
\cite{gallo_contucci_CW}.
To define our model, we partition the set $\{1, \ldots, N\}$ for  $N$ even into a set $S\subset \{1, \ldots, N\}$ with $|S|=\frac{N}{2}$ and  its complement $S^c$. This segmentation induces a
partitioning of the binary hypercube $\{-1,+1\}^N, N\in \N$ the state space of the {\it Ising block model}.
For $\beta>0$ and $0\le \alpha \le \beta$ the model we will consider is defined
by the Hamiltonian
\[
H_{N,\alpha, \beta, S}(\sigma):= -\frac \beta {2N} \sum_{i \sim j} \sigma_i \sigma_j -\frac \alpha {2N} \sum_{i \not\sim j} \sigma_i \sigma_j, \quad \sigma \in \{-1,+1\}^N.
\]
Here we write $i \sim j$, if either $i, j \in S$ or $i,j \in S^c$ and $i \not \sim j$ otherwise.
This Hamiltonian induces a Gibbs measure
\begin{equation}\label{Hamiltonian}
\mu_{N, \alpha, \beta} (\sigma)=\mu_{N, \alpha, \beta,S} (\sigma):= \frac{e^{-H_{N,\alpha, \beta}(\sigma)}}{\sum_{\sigma'}e^{-H_{N,\alpha, \beta}(\sigma')}}=:
\frac{e^{-H_{N,\alpha, \beta}(\sigma)}}{Z_{N, \alpha,\beta}}.
\end{equation}
A closely related version of this model has been investigated in \cite{werner_curie_weiss}. However, the couplings in \cite{werner_curie_weiss} between the blocks have the same strength of interaction as the couplings within a block. We were informed that a more general version of the model will be studied in
\cite{werner_curie_weiss2}.
Similar to the Curie-Weiss model the Ising block model has an order parameter: the vector of block magnetizations, $m:=m^N:=(m^N_1, m^N_2)$, where
\[
m_1:=m^N_1:=m_1(\sigma):= \frac {2} {  N} \sum_{i \in S} \sigma_i \qquad \mbox{and } \quad  m_2:=m^N_2:=m_2(\sigma):= \frac {2} {N} \sum_{i \notin S} \sigma_i.
\]
Indeed, the Hamiltonian is handily rewritten as
\begin{equation*}
H_{N,\alpha, \beta, S}(\sigma)= -\frac N 2 \left(\frac{1}{2}\alpha m_1 m_2 +\beta \frac{1}{4} m_1^2+ \frac{1}{4} \beta m_2^2\right).
\end{equation*}
This observation is not only a convenient way to analyze the block spin Ising model, it also makes $m$ an obvious choice to describe its behaviour and its phase transitions. To characterize them, recall that with the above notation for $\alpha  =\beta$ one reobtains the Curie-Weiss or mean-field Ising model at inverse temperature $\beta$, i.e. the model on $\{-1,+1\}^N$ given by $H_{CW}(\sigma)= \frac 1{2N}\sum_{i,j} \sigma_i \sigma_j$ and Gibbs measure
$
\mu^{CW}_{N,\beta}(\sigma)=
\frac{e^{-\beta H_{CW}(\sigma)}}{Z^{CW}_{N,\beta}}.
$
Also, recall (\cite{Ellis-EntropyLargeDeviationsAndStatisticalMechanics}) that the Curie-Weiss model undergoes a phase transition at $\beta=~1$. This phase transition can be described by saying that the distribution of the parameter $\overline{m} = \frac{1}{N}\sum_i \sigma_i$ (also called the magnetization) weakly converges to the Dirac measure in 0, $\delta_0$, if $\beta \le 1$ while it converges to the mixture $\frac 12 (\delta_{m^+(\beta)}+ \delta_{-m^+(\beta)})$, if $\beta >1$. Here $m^+(\beta)$ is the largest solution of the so-called Curie-Weiss equation
$
z = \tanh(\beta z).
$
A similar result was proven for the block spin Ising model  in \cite{BRS17_blockmodel} (the authors also allow for negative values of $\alpha$). There the authors (implicitly) show
\begin{theorem}cf. \cite[Proposition 4.1]{BRS17_blockmodel}\label{theo1}
In the above assume that $0\le \alpha < \beta$ and denote by $\rho_{N,\alpha,\beta}$ the distribution of $m$ under the Gibbs measure $\mu_{N, \alpha, \beta} $. Then
\begin{itemize}
\item If $\beta +\alpha \le 2$, then $\rho_{N,\alpha,\beta}$ weakly converges to the Dirac measure in $(0,0)$.
\item If $\beta +\alpha > 2$ and $\alpha=0$ then $\rho_{N,\alpha,\beta}$ weakly converges to the mixture of Dirac measures $\frac 1 4 \sum_{ s_1, s_2 \in \{-,+\}} \delta_{(s_1 m^+(\beta/2), s_2 m^+(\beta/2))}$.
\item If $\beta +\alpha > 2$ and $\alpha>0$ then $\rho_{N,\alpha,\beta}$ weakly converges to the mixture of Dirac measures
$\frac 1 2 (\delta_{(m^+(\frac{\alpha+\beta}2),  m^+(\frac{\alpha+\beta}2))}+\delta_{(-m^+(\frac{\alpha+\beta}2),  -m^+(\frac{\alpha+\beta}2)}) $.
\end{itemize}
\end{theorem}
Theorem \ref{theo1} is the trigger for another question: In the Curie-Weiss model the phase transition can also be observed on the level of fluctuations of the magnetization: As is shown in \cite{Ellis_Newman_78a}, \cite{Ellis_Newman_78b} or in
\cite[Theorems V.9.4 and V.9.5]{Ellis-EntropyLargeDeviationsAndStatisticalMechanics}
or \cite{EL10}, for $\beta <1$ the parameter $\sqrt N \overline{m}$ obeys a standard CLT with expectation 0 and variance $\frac 1 {1-\beta}$, while for $\beta=1$ one has to scale differently: Then $N^{1/4} \overline{m}$ converges in distribution to a random variable that has Lebesgue density proportional to $\exp(-\frac{1}{12} x^4)$. Our question in this note is, whether a similar behaviour can be observed for the block spin Ising model and how the limit distribution depends on the relation between $\alpha$ and $\beta$.
To answer this question we show 
\begin{theorem}\label{CLT_m}
For the block spin Ising model assume that $0 \le \alpha < \beta$ and that $\beta+ \alpha < 2$ . Then,
$\sqrt N m$ converges in distribution to a 2-dimensional Gaussian random variable with expectation 0 and covariance matrix
$
\Sigma=s^2 \begin{pmatrix}
1 & r\\
r & 1
\end{pmatrix}
$
with $s^2= \frac{8-4\beta}{(2-\beta)^2-\alpha^2}$ and $r=\frac{\alpha}{2-\beta}$.
\end{theorem}

\begin{remark}
It is well known that in the standard Curie-Weiss a CLT also holds in the presence of an external field \cite{{Ellis_Newman_78b}}, i.e.~with a Hamiltonian of the form
$H_{CW}(\sigma)= \frac 1{2N}\sum_{i,j} \sigma_i \sigma_j+h \sum_i \sigma_i$ with $h>0$. We are firmly convinced that a similar result is true for our model as well. However, we did not try to prove it. Finally, one might ask, whether -- in the spirit of \cite{EL10} -- Stein's method may be applied to our situation as well. We consider this a more challenging question, because a multi-dimensional version of Stein's method would be needed. We may consider this problem in a different paper.
\end{remark}

On the other hand, if  $\beta+ \alpha = 2$ the fluctuations are no longer Gaussian
\begin{theorem}\label{CLT_m_crit}
For the block spin Ising model assume that $0 \le \alpha < \beta$ and that \mbox{$\beta+\alpha=2$}. Then,
$N^{\frac 14} m_1$ converges in distribution to a probability measure $\rho$ on $\R$. The measure $\rho$ is absolutely continuous with Lebesgue-density
$g(x) = \exp(-\frac{1}{12} x^4)/Z$, where $Z$ is a normalizing constant. The difference between $ m_1$ and $m_2$ multiplied by $\frac{\sqrt N}{2}$, i.e. $\tilde m_1 - \tilde m_2:= \frac{\sqrt N}{2} (m_1 - m_2)$, however, is asymptotically Gaussian with mean $0$ and variance $\frac 2 {2-(\beta-\alpha)}$.
\end{theorem}

We will show Theorems \ref{CLT_m} and \ref{CLT_m_crit} in Section 3 using a Hubbard-Stratonovich transformation for an appropriate function of $m$. Before, in Section 2, however, we will give an alternative proof of Theorem \ref{theo1} using the theory of large deviations. This is not only interesting in its own right, but also provides a way to derive limit theorems in more complicated settings, see e.g \cite{loeweschubertvermet}.

\section{An LDP for the vector of block magnetizations}
Differing from the line of arguments in \cite{BRS17_blockmodel}, Theorem \ref{theo1} can also be shown by proving a large deviation principle (LDP) for $0\le \alpha\leq \beta$. To this end, we will slightly change our variables and consider the vector $v=(v_1, v_2)$ with
$v_1:= \frac{1}{2} m_1$ and $v_2:= \frac{1}{2} m_2$. We show
\begin{theorem}\label{LDP}
For every $S\subset \{1, \ldots, N\}$  with $|S|=\frac{N}{2}$ the vector $v$ obeys a principle of large deviations (LDP) under the Gibbs measure $\mu_{N,\alpha,\beta}:=\mu_{N,\alpha,\beta,S}$, with speed $N$ and rate function
$J_v(x):= \sup_{y \in \R^2} [F_v(y)-J(y)]-[F_v(x)-J(x)].$
Here $F_v:\R^2 \to \R$ is defined by
\begin{equation}\label{contraction}
F_v(x) := \frac 12 \left(\beta x_1^2+ \beta  x_2^2+ 2\alpha  x_1  x_2\right)
\end{equation}
and for $x \in \R^2$
\[
J(x):= \sup_{t \in \mathbb R^2} \left[ \langle t,x \rangle - \frac{1}{2} \log \cosh(t_1)- \frac{1}{2} \log \cosh(t_2)\right].
\]
This implies that the convergence in Theorem \ref{theo1} (for $0\le\alpha\le \beta$) is exponentially fast.
\end{theorem}
\begin{proof}
We will prove this theorem in two steps, first we show the LDP, then, how one derives Theorem \ref{theo1} from it.

\noindent
\underline{Step 1:}

\noindent
 First, note that the case $\alpha=0$ is trivial. Then, the system consists of two independent Curie-Weiss models on $\frac N2$ spins at temperature $\beta$. The LDP for the magnetization in each of the systems is known (cf.~e.g.~\cite{Ellis-EntropyLargeDeviationsAndStatisticalMechanics}) and transferring these LDPs to the vector $v$ (with independent components) is trivial.
We will thus assume that $\alpha>0$.

Let us consider the moment generating function of the vector $v$. To this end let $t=(t_1, t_2)\in \R^2$. Then
the moment generating function of $v$ in $t$ is given by
\[
\E \exp(N \langle t, v\rangle)= \cosh(t_1)^{\frac{N}{2}} \cosh(t_2)^{ \frac{N}{2}},
\]
where here $\mathbb E$ denotes the expectation with respect to the  a priori
measure  $\frac{1}{2}(\delta_{-1}-\delta_{+1})$.
This readily yields
$\lim_{N \to \infty} \frac 1N \log \E \exp(N \langle t, v\rangle)= \frac{1}{2} \log \cosh(t_1)+ \frac{1}{2} \log \cosh(t_2).$
As the right hand side of this expression is finite and differentiable on all of $\R^2$,
by the G\"artner-Ellis Theorem \cite[Theorem 2.3.6]{dembozeitouni} this computation implies an LDP for $v$ under the uniform distribution with speed $N$ and rate function
\[
J(x):= \sup_{t \in \R^2} \left[ \langle t,x \rangle - \frac{1}{2} \log \cosh(t_1)- \frac{1}{2} \log \cosh(t_2)\right]= \frac{1}{2} I\left(2x_1\right)+ \frac{1}{2} I\left(2 x_2\right)
\]
for $x \in \R^2$.
Here
$
I(x):= \frac 12 (1+x) \log (1+x) + \frac 12 (1-x) \log (1-x).
$
Now the Hamiltonian $H_{N,\alpha, \beta, S}(\sigma)$  of our model can also be rewritten in terms of $v$:
\[
H_{N,\alpha, \beta, S}(\sigma)= -\frac N2 \left(\beta v_1^2(\sigma)+ \beta  v_2^2(\sigma)+ 2\alpha  v_1(\sigma)  v_2(\sigma)\right).
\]
This fact, together with the above LDP and the exponential form of the Gibbs measure and the LDP for integrals of exponential functions (see e.g. \cite[Theorem III.17]{den_hollander_large_deviations} -- a direct consequence of Varadhan's Lemma \cite[Theorem 4.3.1]{dembozeitouni} -- implies that the distribution of $v$ under $\mu_{N,\alpha,\beta,S}$ satisfies an LDP with speed $N$ and rate function
\[
J_v(x):= \sup_{y \in \R^2} [F_v(y)-J(y)]-[F_v(x)-J(x)],
\]
where $F_v:\R^2 \to \R$ is given by \eqref{contraction}. A change of variables yields the desired LDP.

\noindent
\underline{Step 2:}

\noindent
If $M$ denotes the set of minima of $J_v$ and $B_\vep(M):= \bigcup_{y \in M} B_\vep(y)$ (where $B_\vep(y)$ are open balls of radius $\vep>0$ centered around $y$) we obtain from the upper bound of the LDP that
\[
\P(v \notin B_\vep(M)) \le \exp\left(-\frac{N}2 \inf_{x \in B_\vep^c(M)} J_v(x)\right)
\]
for $N$ large enough. The $\inf$ on the right hand side of the inequality is positive. In this sense,
$v$ concentrates in the minima of $J_v$ exponentially fast. By a change of variables, again, this implies that $m$ concentrates exponentially fast in the (global) minima of $J_m$ defined by
\[
J_m(x):= \sup_{y \in \R^2} [F_m(y)-\tilde J(y)]-[F_m(x)-\tilde J(x)],
\]
\[
\tilde J(x):= \frac{1}{2} I(x_1)+ \frac{1}{2} I(x_2) \quad
\text{ and } \quad
F_m(x) := \frac 12 \left(\beta \frac{1}{4} x_1^2+ \beta \frac{1}{4} x_2^2+ \frac{1}{2}\alpha  x_1  x_2\right).
\]
The minima of $J_m$ are the maxima of $F_m(x)-\tilde{J}(x)$. These
satisfy $\nabla (F_m-\tilde{J})(x)=0,$ i.e.
\begin{equation}\label{critical_point1}
\frac{1}{2} \beta x_1 + \frac{1}{2}\alpha x_2 = \artanh(x_1) \qquad \mbox{and}\qquad
\frac{1}{2} \beta x_2 + \frac{1}{2} \alpha x_1 = \artanh(x_2).
\end{equation}
Note that the vector $(0,0)$ is always a solution to this system of equations and hence a critical point of  $F_m-\tilde{J}$.

We start with the high temperature regime, i.e.~we consider $\beta +\alpha < 2$.
By an easy calculation we find that the Hesse matrix of $F_m(x,y)-\tilde{J}(x,y)$ is given by
 \[\frac{1}{2}
 \begin{pmatrix}
 \frac{1}{2} \beta - \frac 1 {1-x^2} & \frac{1}{2}\alpha\\
 \frac{1}{2}\alpha &  \frac{1}{2} \beta - \frac 1 {1-y^2}
 \end{pmatrix}.
 \]
Hence, the Hesse matrix in the point $(0,0)$ is negative definite, i.e. $(0,0)$ is a local maximum of $F_m(x)-\tilde{J}(x)$, if  $0 \le \alpha  \le \beta <2$ and
$
\left( 1-\frac{1}{2} \beta \right)^2 - \frac{1}{4} \alpha^2>0.
$
This is true, if $\alpha + \beta < 2$.

Next, we will see that, in this case, the point $(0,0)$ is the only solution to the system of equations in  \eqref{critical_point1}, and hence the global maximum of  $F_m(x)-\tilde{J}(x)$.
To this end, we rewrite the equations in \eqref{critical_point1}  as
\begin{equation*}
x_2= \frac{2}{\alpha} \left( \artanh (x_1) - \frac{1}{2} \beta x_1 \right) \quad \text{and} \quad
x_1= \frac{2}{\alpha} \left( \artanh (x_2) - \frac{1}{2} \beta x_2 \right).
\end{equation*}
Hence, for $|x|< 1$ and
$f(x):= \frac{2}{\alpha}\left( \artanh (x) - \frac{1}{2} \beta x \right) $ we have
\begin{equation*}
x_1=f(x_2), \quad x_2=f(x_1) \quad \text{resp.} \quad x_1=f^2(x_1), \quad x_2=f^2(x_2).
\end{equation*}
This means we are looking for the fixed points of $f^2$.
We note that for $0\leq \beta \leq 2$, we have for all $|x|<1$
\begin{equation*}
f'(x)=\frac{2}{\alpha} \left( \frac{1}{1-x^2} - \frac{1}{2} \beta \right) >0
\end{equation*}
and hence $f$ is invertible. The fixed points of $f^2$ are thus the same as the fixed points of $(f^{-1})^2$ resp.~of $(f^{-1}).$
We see that $f^{-1}$ is a strong contraction for $\alpha + \beta < 2$ from $(f^{-1})'=\frac{1}{f'}$ and for all $y \in (-1,1)$
\begin{equation*}
\frac{1}{f'(y)}=\frac{\alpha}{2}\frac{1}{ \left( \frac{1}{1-y^2} - \frac{1}{2} \beta \right)} \leq \frac{\alpha}{2}\frac{1}{ \left( 1 - \frac{1}{2} \beta \right)} \le 1-\vep,
\end{equation*}
for $\vep>0$ small enough.
Thus, by Banach's fixed point theorem, there is a single fixed point, which has to be equal to zero.

Next consider the critical line $\beta + \alpha =2$. Here the arguments are almost the same. The only difference is, that now $\frac{1}{f'(y)}\big|_{y=0}=1$, while $\frac{1}{f'(y)}<1$ for all other $y$. Hence $f^{-1}$ is a weak contraction for $\alpha + \beta = 2$. However, the magnetizations $m_1$ and $m_2$ live on the compact interval $[-1,1]$, such that we can again conclude that $f^{-1}$ has the unique fixed point $0$.

Now we consider $\alpha + \beta >2$. In this case $(0,0)$ is still a solution to \eqref{critical_point1}. However, in this case, it is either a saddle point or a local minimum of $F_m-\tilde{J}$, because it is not a maximum. Indeed, choose $x=y$, i.e.
\begin{equation*}
F_m(x,x) -\tilde J(x,x)= \frac 12 \left( \frac{\beta + \alpha}{2}     \right)x^2 -I(x).
\end{equation*}
From the one dimensional Curie-Weiss model we know that for $\alpha + \beta >2$ the maximum at attained away from zero. Hence, $(0,0)$ is not a maximum of $F_m-\tilde{J}$.
Recalling that $\alpha >0$, we see directly from the definition of $F_m$ and $\tilde{J}$ that a point $(x,y)$ can only be a maximum, if $x$ and $y$ have the same sign.
We will see that $f$ and thus $f^2$ has exactly one positive fixed point $m^*$ and one negative fixed point $-m^*$.
This again is shown using a fixed point argument for $f^{-1}$. Note that now $\frac{1}{f'(0)}>1$. However, the function $y \mapsto \frac{1}{f'(y)}$ is always non-negative on $[0,1]$, it is decreasing, and depends continuously  on $y$ and by the intermediate value theorem there is $y_0$ such that $\frac{1}{f'(y)}\le 1$ for all $y \in [y_0,1]$.
 Thus, $f^{-1}$ restricted to $[y_0,1]$ is a weakly contracting self-map and therefore has a unique fixed point. But this fixed point of $f^{-1}$ is the fixed point of $f$ and is easily checked to satisfy
\begin{equation*}
\tanh \left( \frac{\alpha + \beta}{2} m^*\right)=m^*.
\end{equation*}
This proves the claim.
\end{proof}

\begin{remark}
In the spirit of \cite{werner_curie_weiss}, one can also allow for other sizes of $S$, i.e. for $0 < \gamma <1$  we can consider sets $S \subset \{1,\ldots, N\}$ with $|S|= \gamma N$. Here, we assume, for simplicity, $\gamma N$ to be an integer. In this case, the Hamiltonian $H_{N,\alpha, \beta, S}$ is the same as in \eqref{Hamiltonian} and the magnetizations are
$m_1(\sigma)= \frac {1} {\gamma N} \sum_{i \in S} \sigma_i,  \mbox{  and  }  m_2(\sigma):= \frac {1} {(1-\gamma) N} \sum_{i \notin S} \sigma_i$.
The Hamiltonian can then be rewritten as
\[
H_{N,\alpha, \beta, S}(\sigma)= -\frac N 2 \left(2 \gamma (1-\gamma) \alpha m_1 m_2 +\beta \gamma^2 m_1^2+ (1-\gamma)^2 \beta m_2^2\right).
\]
We believe that a result analogue to Theorem \ref{theo1} can be shown by generalizing the large deviation techniques in the proof of Theorem \ref{LDP}. In the same spirit as in Theorem \ref{CLT_m} and Theorem \ref{CLT_m_crit}, one can also show a Central Limit Theorem for this generalized setting. The technical problems are, however, more demanding. We will return to these questions in a later publication.
\end{remark}
\section{Proof of Theorem \ref{CLT_m} and \ref{CLT_m_crit}}
The proofs of Theorems \ref{CLT_m} and \ref{CLT_m_crit} rely on the same idea. We will first prove limit theorems for two other parameters, that are closely related to $m_1$ and $m_2$. To this end we introduce the random variables
\[
w_1:= w_1(\sigma):=\frac 1 N \sum_{i=1}^N \sigma_i \qquad \mbox{and} \quad w_2:= w_2(\sigma):= \frac 1 N\left(\sum_{i \in S} \sigma_i -\sum_{i \notin S} \sigma_i\right)
\]
and the corresponding standardized versions
\[
\tilde w_1:= \sqrt N w_1 \qquad \mbox{and } \quad \tilde w_2:= \sqrt N w_2.
\]
Note that $m_1 = w_1+w_2$ and $m_2 = w_1-w_2$ and thus limit theorems for $w=(w_1,w_2)$ will imply limit theorems for $m$ and vice versa.

Again, note that the Hamiltonian $H_{N,\alpha, \beta, S}$ can also be rewritten in terms of the variables $w_1$ and $w_2$ resp.~in terms of $\tilde{w}_1$ and $\tilde{w}_2$ as
\begin{equation*}
H_{N,\alpha, \beta, S}(\sigma)= - \frac{N}{2} \left(2 \alpha \frac14 m_1 m_2 + \beta \frac14 m_1^2 + \beta \frac14 m_2^2\right)
=-\frac{1}{4} ((\alpha + \beta) \tilde{w_1}^2 + (\beta-\alpha) \tilde{w_2}^2).
\end{equation*}
Next we will show a Central Limit Theorem for the vector $\tilde w:=(\tilde w_1, \tilde w_2)$ in the high temperature region $0 \le \alpha < \beta < 2$ and
$\alpha+\beta <2$.
\begin{lemma}\label{lemma_w}
Assume that $0 \le \alpha < \beta$ and $\beta +\alpha < 2$. Then, as $N \to \infty$, under the Gibbs measure $\mu_{\alpha, \beta, S}$ the vector $\tilde w$ converges to a 2-dimensional Gaussian distribution with expectation $0$ and covariance matrix
$
\Sigma=\begin{pmatrix}\frac{1}{1-\frac{\alpha+\beta}{2}} &0 \\ 0&  \frac{1}{1-\frac{ \beta - \alpha}{2}}\end{pmatrix}.
$
\end{lemma}
\begin{proof}
Our principal strategy consists of computing a suitable Hubbard-Stratonovich transform of our measure of interest (as e.g. in \cite{gentzloewe}) and expanding it. To this end, let ${\mathcal N}(0, C)$ denote a two-dimensional Gaussian distribution with expectation $0$ and covariance matrix $C$ given by
$
C=  \begin{pmatrix}
\frac{2}{\beta+ \alpha} & 0\\
0 & \frac{2}{\beta-\alpha}
\end{pmatrix}.
$
We will now compute the density of
$\chi_{N,\alpha,\beta} :=
\mu_{N,\alpha,\beta}(\tilde{w})^{-1} * {\mathcal N}(0, C)$:
Let $A$ be a Borel subset of ${\mathbb R}^2$ and let $\varrho_{N,\beta}(\sigma)$ denote the density of $ \mu_{N,\alpha, \beta}$. Then
\begin{align*}
&  \chi_{N,\alpha, \beta} (A) = (\tilde{w})^{-1}
      * {\mathcal N}(0, C) (A)
 =
\sum_{\sigma \in \{-1,1\}^N}
      {\mathcal N}(0,C) (A-\tilde{w})
      \mu_{N,\alpha,\beta}(\sigma)
 \\
 = & K_1
\sum_{\sigma \in \{-1,1\}^N}
\int_{A-\tilde{w} }    \exp \left( -\frac{1}{2}\left( \frac{\alpha + \beta}{2}x^2  +  \frac{\beta - \alpha}{2} y^2 \right)\right) \varrho_{N,\beta}(\sigma) dx dy
\\
 = & K_1
\sum_{\sigma \in \{-1,1\}^N}
\int_{A}    \exp \left( -\frac{1}{2}\left( \frac{\alpha + \beta}{2}(x- \tilde{w_1})^2  +  \frac{\beta - \alpha}{2} (y-\tilde{w_2})^2\right)\right) \varrho_{N,\beta}(\sigma) dx dy
\\
 = & K_2
\sum_{\sigma \in \{-1,1\}^N}
\int_{A}    \exp \left( -\frac{1}{2}\left( \frac{\alpha + \beta}{2}(x- \tilde{w_1})^2  +  \frac{\beta - \alpha}{2} (y-\tilde{w_2})^2\right)\right)\\
&\hskip 3cm \times \exp\left(\frac{1}{4} \left((\alpha + \beta) \tilde{w_1}^2 + (\beta-\alpha) \tilde{w_2}^2\right)\right) dx dy
\\
 = & K_2
\sum_{\sigma \in \{-1,1\}^N}
\int_{A}    \exp \left( -\frac{\alpha + \beta}{4}  x^2 -   \frac{  \beta - \alpha }{4} y^2 \right)
\exp\left(
\frac{\alpha + \beta}{2}  x \tilde{w_1} + \frac{\beta  - \alpha }{2}  y \tilde{w_2}\right)dx dy
\\
 = & K_3
\int_{A}    \exp\left( -\frac{\alpha + \beta}{4}  x^2 -   \frac{   \beta - \alpha}{4} y^2 \right)\exp\left( \frac N2 \log \cosh \left(\frac{1}{\sqrt{N}} \left(\frac{\alpha + \beta}{2} x + \frac{\beta- \alpha}{2} y\right)\right)\right.\\
 & \qquad \qquad +  \frac N2 \log \cosh \left(\frac{1}{\sqrt{N}} \left( \frac{\alpha + \beta}{2} x - \frac{\beta-\alpha}{2} y\right)\right)dx dy
\end{align*}
Here, we used
$
K_1:= \frac 1 {2 \pi \sqrt{\mathrm{det}\,C}},   K_2 := \frac 1 {{2 \pi \sqrt{\mathrm{det}\,C} Z_{N,\alpha,\beta}}}
  \mbox{and } \quad  K_3:=K_2 2^N.
$
Denote by
\begin{align*}
\Phi(x,y):=&\Phi_{N,\alpha,\beta,S}(x,y)
\\
:= & \frac{1}{4} (\alpha + \beta) x^2 +   \frac{1}{4} (  \beta - \alpha) y^2-
 \frac N2 \log \cosh \left(\frac{1}{\sqrt{N}} \left(\frac{\alpha + \beta}{2} x + \frac{\beta- \alpha}{2} y\right) \right)
 \\
 &
\quad \quad - \frac N2 \log \cosh \left(\frac{1}{\sqrt{N}}\left( \frac{\alpha + \beta}{2} x - \frac{\beta-\alpha}{2} y\right)\right).
\end{align*}
Now recall the second order Taylor expansion
$\log \cosh (x) = \frac{1}{2}x^2+ \mathcal{O}(x^4)$.
Thus
\begin{align*}
\Phi(x,y)=&\frac{1}{4} (\alpha + \beta) x^2 +   \frac{1}{4} (  \beta - \alpha) y^2
- \frac{1}{4}\left(  \frac{(\alpha + \beta)^2}{4} x^2 +  \frac{\beta^2- \alpha^2}{2}xy  + \frac{(\beta-\alpha)^2}{4}y^2\right)
\\
& \hskip +2cm - \frac{1}{4} \left(  \frac{(\alpha + \beta)^2}{4} x^2 -  \frac{\beta^2- \alpha^2}{2}xy  + \frac{(\beta-\alpha)^2}{4}y^2\right)+\mathcal{O}(N^{-1}),
\\
=& \frac{x^2}{2} \left( \frac{\alpha + \beta}{2} - \left( \frac{\alpha + \beta}{2}\right)^2\right)+
\frac{y^2}{2} \left( \frac{\beta-\alpha }{2} - \left( \frac{  \beta-\alpha }{2}\right)^2\right) +\mathcal{O}(N^{-1}),
\end{align*}
where the constant in the $\mathcal{O}(N^{-1})$-term depends on $x$ and $y$. However, the convergence is uniform on compact subsets of $\R^2$.
Thus
\begin{align*}
&\chi_{N,\alpha, \beta} (A)
= K_3\int_A  e^{-\Phi(x,y)} dx\,  dy\\
 =& K_3 \int_A\exp  \left[ -\frac{x^2}{2} \left(  \frac{\alpha + \beta}{2} -\left( \frac{\alpha + \beta}{2}\right)^2 \right)-   \frac{y^2}{2} \left( \frac{  \beta - \alpha}{2} -\left( \frac{  \beta -\alpha}{2}\right)^2 \right)+\mathcal{O}(N^{-1})\right]dx \,dy
\end{align*}
and the convergence in the $\mathcal{O}(N^{-1})$-term is uniform on compact subsets of $\R^2$.

To turn this into a weak convergence statement, we need to control integrals over unbounded sets as well, in particular, we need to treat the case $A=\R^2$ to see that $K_3$ converges to $\frac 1 {2 \pi \sqrt{\mathrm{det} \Sigma'}}$ for
\begin{equation}
\label{eq:def_sigma_prime}
\Sigma':= \begin{pmatrix} \left[  \frac{\alpha + \beta}{2} -\left( \frac{\alpha + \beta}{2}\right)^2 \right]^{-1} & 0 \\ 0 &  \left[ \frac{  \beta - \alpha}{2} -\left( \frac{  \beta -\alpha}{2}\right)^2 \right]^{-1}\end{pmatrix}.
\end{equation}
Hence, for any measurable set $A \subset \R^2$ we  write
\begin{align}
&\int_A \Psi(x,y)dx\, dy \nonumber
\\
&= \int_{A \cap B(0,R)}\Psi(x,y)dx\, dy+ \int_{A \cap B(0,R)^c \cap B(0,r \sqrt N)}\Psi(x,y)dx\, dy+ \int_{A \cap B(0,r \sqrt N)^c} \Psi(x,y)dx\, dy\label{split}
\end{align}
where we set
\begin{align*}
\Psi(x,y):= e^{- \Phi(x,y)}.
\end{align*}
Here for any $l >0$ we denote by $B(0,l)$ the ball in $\R^2$ with center in $0$ and radius $l$. Further, we consider numbers $R>0$ and $r>0$ and we will send $R$ to $\infty$ and consider $r$ sufficiently small. We will refer to the summands on the right hand of \eqref{split} as inner region, intermediate region and outer region, respectively. The goal is to see that the inner region contributes all mass to the integral as $R \to \infty$.

As already marked above for fixed $R>0$
\begin{align*}
\lim_{N\to \infty}& \int_{A \cap B(0,R)}  \Psi(x,y) dx\, dy \nonumber\\
=& \int_{A \cap B(0,R)} \exp  \left[ -\frac{1}{2} x^2\left(  \frac{\alpha + \beta}{2} -\left( \frac{\alpha + \beta}{2}\right)^2 \right)
-   \frac{1}{2} y^2\left( \frac{  \beta - \alpha}{2} -\left( \frac{  \beta -\alpha}{2}\right)^2 \right)\right]dx \,dy.
\end{align*}
Next, we treat the outer region, i.e.~$A \cap B(0,r \sqrt N)^c$. 
By the easy estimate $\log \cosh(x) \leq \frac{1}{4}x^2+1$, we see that there is a positive constant $C_{\alpha, \beta}$ such that 
\begin{align*}
-\Phi(x,y) \leq -C_{\alpha, \beta}(x^2+y^2) +N=-\frac{C_{\alpha, \beta}}{2}(x^2+y^2) + (N-\frac{C_{\alpha, \beta}}{2} (x^2+y^2)).
\end{align*}
For $r_0:= \sqrt{\frac{2}{C_{\alpha, \beta}}}$ this estimate implies by the dominated convergence theorem
\begin{equation}\label{outer_region_1}
\lim_{N \to \infty} \int_{A \cap B(0,r_0 \sqrt N)^c}  \Psi(x,y) dx\, dy
 =0
\end{equation}

For $A \cap B(0,r_0 \sqrt N)   \cap B(0,r  \sqrt N)^c $ for some $r<r_0$ let us rewrite the exponent   as
\[
\Phi(x,y) = N \tilde \Phi\left(\frac x {\sqrt N}, \frac y {\sqrt N} \right)
\]
where
\begin{align*}
\tilde \Phi(x,y)
:= & \frac{1}{4} (\alpha + \beta) x^2 +   \frac{1}{4} (  \beta - \alpha) y^2 \\
&\quad -
 \frac 12 \log \cosh \left(\frac{\alpha + \beta}{2} x + \frac{\beta- \alpha}{2} y\right)
- \frac 12 \log \cosh \left( \frac{\alpha + \beta}{2} x - \frac{\beta-\alpha}{2} y\right).
\end{align*}
Analyzing $\tilde \Phi$ we see that it becomes minimal only if, $\nabla \tilde \Phi=0$ and
\[
\nabla \tilde \Phi(x,y) = \left( \begin{array}{l}
\frac 12 c_1 x - \frac{c_1}4 \tanh(\frac{c_1 x+c_2 y}2)-\frac{c_1}4 \tanh(\frac{c_1 x-c_2 y}2)\\
\frac 12 c_2 y - \frac{c_2}4 \tanh(\frac{c_1 x+c_2 y}2)+\frac{c_2}4 \tanh(\frac{c_1 x-c_2 y}2)
\end{array}
\right),
\]
where we abbreviate $c_1:= \alpha+\beta$ and $c_2:=\beta-\alpha$. This means, we aim to solve
\begin{align*}
x & = \frac 12 \tanh\left(\frac{c_1 x+c_2 y}2\right)+\frac{1}2 \tanh\left(\frac{c_1 x-c_2 y}2\right)\\
y &= \frac{1}2 \tanh\left(\frac{c_1 x+c_2 y}2\right)-\frac{1}2 \tanh\left(\frac{c_1 x-c_2 y}2\right).
\end{align*}
This is done in the spirit of the arguments in Section 2. Indeed, denoting by
\[
G(x,y):= \left( \begin{array}{l}
\frac 12 \tanh(\frac{c_1 x+c_2 y}2)+\frac{1}2 \tanh(\frac{c_1 x-c_2 y}2)\\
\frac{1}2 \tanh(\frac{c_1 x+c_2 y}2)-\frac{1}2 \tanh(\frac{c_1 x-c_2 y}2)
\end{array}
\right),
\]
we see that its Jacobian is given by
\[
J_G(x,y) = \left( \begin{array}{ll}
\frac {c_1}{4} \left( \frac 1 {\cosh^2(\frac{c_1 x+c_2 y}2)}+\frac 1 {\cosh^2(\frac{c_1 x-c_2 y}2)} \right)&
\frac {c_2}{4} \left( \frac 1 {\cosh^2(\frac{c_1 x+c_2 y}2)}-\frac 1 {\cosh^2(\frac{c_1 x-c_2 y}2)}  \right)\\
\frac {c_1}{4} \left( \frac 1 {\cosh^2(\frac{c_1 x+c_2 y}2)}-\frac 1 {\cosh^2(\frac{c_1 x-c_2 y}2)} \right)&
\frac {c_2}{4} \left( \frac 1 {\cosh^2(\frac{c_1 x+c_2 y}2)}+\frac 1 {\cosh^2(\frac{c_1 x-c_2 y}2)}  \right)
\end{array}
\right).
\]
This means that
$
||J_G(x,y)||_1 \le \max \left( \frac{c_1}{2}, \frac{c_2}{2}\right)= \frac{c_1}2 <1,
$
where $\| \cdot\|_1$ denotes the  maximum absolute column sum of a matrix.
Thus $G$ is a (strict) contraction with fixed point $(0,0)$ and  hence $\tilde \Phi$ is minimal in $(0,0)$. Therefore, for every $r >0$.
\[
M_r:= \inf_{(x,y) \notin B(0,r)} \tilde \Phi(x,y) >0.
\]
This implies that
\begin{equation}\label{outer_region_2}
\lim_{N \to \infty} \int_{A \cap B(0,r \sqrt N)^c\cap  B(0,r_0 \sqrt N)  }  \Psi(x,y) dx\, dy
  =  \lim_{N \to \infty}e^{-NM_r} \text{Vol}(B(0,r_o\sqrt{N})) 
 =0
\end{equation}
for any $r>0$. Combining \eqref{outer_region_1} and \eqref{outer_region_2}, we see that  the outer region is asymptotically negligible.

Let us turn to the intermediate region.
Here we  take again a Taylor expansion of the $\log \cosh$ on an interval $[-z_0, z_0]$, $z_0>0$, around the origin to first order with a Lagrange bound on the remainder:
$
\log \cosh (z)= \frac {z^2} 2 + C  z^4
$
with a constant $C$ that depends on $z_0$.
This yields for $\frac{x}{\sqrt{N}}, \frac{y}{\sqrt{N}} \in B(0,r)$
\begin{align*}
&N \tilde \Phi\left(\frac x{\sqrt N},\frac y {\sqrt N}\right)
=
\frac{x^2}{2} \left( \frac{\alpha + \beta}{2}\right)+\frac{y^2}{2} \left( \frac{\beta-\alpha }{2}\right)
\\
&\quad  \quad -\frac N4 \left(\frac{(\alpha + \beta)x}{2\sqrt N} + \frac{(\beta-\alpha) y }{2\sqrt N}\right)^2
-\frac N4 \left(\frac{(\alpha + \beta)x}{2\sqrt N} - \frac{(\beta-\alpha) y }{2\sqrt N}\right)^2\\
& \quad  \quad
- C_r N \left[\left(\frac{(\alpha + \beta)x}{2\sqrt N} + \frac{(\beta-\alpha) y }{2\sqrt N}\right)^4
+\left(\frac{(\alpha + \beta)x}{2\sqrt N} - \frac{(\beta-\alpha) y }{2\sqrt N}\right)^4\right]\\
&\ge
\frac{x^2}{2} \left( \frac{\alpha + \beta}{2} - \left( \frac{\alpha + \beta}{2}\right)^2\right)+
\frac{y^2}{2} \left( \frac{\beta-\alpha }{2} - \left( \frac{  \beta-\alpha }{2}\right)^2\right)  -2 C_r N \left( \frac {|x|+|y|}{\sqrt N}\right)^4,
\end{align*}
where we used $\frac{\alpha + \beta}{2} \leq 1$.
However, on $A_R^r:= A \cap B(0,R)^c \cap B(0,r \sqrt N)$ we can estimate the last line by
\begin{equation*}
 2 C_r N \left( \frac {|x|+|y|}{\sqrt N}\right)^4 = 2 C_r   \left( |x|+|y|\right)^2 \left( \frac {|x|+|y|}{\sqrt N}\right)^2 \leq 8 C_r (x^2 + y^2) r^2 =: \tilde{C_r} r^2 (x^2 + y^2)
\end{equation*}
for $\tilde C_r$ uniformly on $B(0, r \sqrt N)$. Note that $\tilde C_r r^2$ depends continuously on $r$ and converges to $0$ as $ r \to 0$. In particular, if $r$ is small enough we have that
$\left( \frac{\alpha + \beta}{2} - \left( \frac{\alpha + \beta}{2}\right)^2\right)-\tilde C_r r^2>0$ as well as
$\left( \frac{\beta-\alpha}{2} - \left( \frac{\beta-\alpha}{2}\right)^2\right)-\tilde C_r r^2>0$.
But for this choice of $r$ and $\tilde C_r$ we arrive at
\begin{align*}
&\int_{A_R^r} \Psi(x,y) \, dx\, dy\nonumber\\
\le & \int_{A_R^r}  \exp  \left[ -\frac{x^2}{2} \left( \frac{\alpha + \beta}{2} - \left( \frac{\alpha + \beta}{2}\right)^2-C_r r^2\right)-
\frac{y^2}{2} \left( \frac{\beta-\alpha }{2} - \left( \frac{  \beta-\alpha }{2}\right)^2-C_r r^2\right)\right]dx  dy
\end{align*}
and the right hand side is an integrable function. Thus for $R \to \infty$ the right hand side as well as the left hand side converges to $0$.

Putting the estimates together, we have seen that $ \chi_{N,\alpha, \beta}$ converges weakly to the 2-dimensional Gaussian distribution with expectation $0$ and covariance matrix
$\Sigma'$, where $\Sigma'$ is given in \eqref{eq:def_sigma_prime}.
This weak convergence is equivalent to the convergence of the characteristic functions.

Computing the characteristic functions of the Gaussian distribution involved in the above proof, we have therefore shown that the characteristic function of $\tilde w$ in the point $t=(t_1, t_2)\in \R^2$ satisfies
\begin{align*}
\lim_{N \to \infty} \mathbb E(e^{it \tilde{w}}) e^{-\frac{1}{2}(\frac{2}{\alpha +\beta})t_1^2 -\frac{1}{2}(\frac{2}{\beta- \alpha})t_2^2 }=e^{-\frac{1}{2}t_1^2 [  \frac{\alpha + \beta}{2} -\left( \frac{\alpha + \beta}{2}\right)^2]^{-1} - \frac{1}{2} t_2^2  [  \frac{  \beta - \alpha}{2} -\left( \frac{  \beta -\alpha}{2}\right)^2]^{-1}}.
\end{align*}
This implies that
\begin{align*}
\lim_{N \to \infty} \mathbb E(e^{it \tilde{w}}) &= e^{\frac{1}{2}(\frac{2}{\alpha +\beta})t_1^2+ \frac{1}{2}(\frac{2}{\beta- \alpha})t_2^2 } e^{-\frac{1}{2}t_1^2 [  \frac{\alpha + \beta}{2} -\left( \frac{\alpha + \beta}{2}\right)^2]^{-1} - \frac{1}{2} t_2^2  [  \frac{  \beta - \alpha}{2} -\left( \frac{  \beta -\alpha}{2}\right)^2]^{-1}}
\\
&=e^{-\frac{1}{2}t_1^2 \left(\frac{1}{1-\frac{\alpha+\beta}{2}}\right)-\frac{1}{2}t_2^2 \left(\frac{1}{1-\frac{ \beta - \alpha}{2}}\right)}.
\end{align*}
Turning this into a weak convergence statement again, we obtain
\begin{equation*}
(\tilde{w_1}, \tilde{w_2}) \xrightarrow{N\to \infty} \mathcal N(0, \Sigma), \quad \Sigma=\begin{pmatrix}\frac{1}{1-\frac{\alpha+\beta}{2}} &0 \\ 0&  \frac{1}{1-\frac{ \beta - \alpha}{2}}\end{pmatrix}
\end{equation*}
in distribution.
\end{proof}

\begin{proof}[Proof of Theorem \ref{CLT_m}]
The proof of Theorem \ref{CLT_m} is straightforward from the above lemma. As observed we have that $m_1=w_1+w_2$ and $m_2=w_1-w_2$, thus $\sqrt N m_1=\sqrt N (w_1+w_2)=\tilde{w_1} + \tilde{w_2}$ as well as $\sqrt N m_2=\sqrt N (w_1-w_2)=\tilde{w_1}-\tilde{w_2}$.
Thus Lemma \ref{lemma_w} gives that $m_1$ and $m_2$ are asymptotically normal with mean 0 and variance
\[
\lim_{N \to \infty}\V(\sqrt N m_1)= \frac 1 {1-\frac{\beta+\alpha}2}+\frac 1 {1-\frac{\beta-\alpha}2}= \frac 2 {2-\beta-\alpha}+\frac 2 {2-\beta+\alpha}= \frac {4 (2-\beta)}{(2-\beta)^2-\alpha^2}.
\]
Moreover, the same considerations together with Lemma \ref{lemma_w} show that their covariance is given by
\[
\lim_{N \to \infty}\mathrm{Cov}(\sqrt N m_1, \sqrt N m_2)= \V \tilde w_1 - \V \tilde w_2= \frac {4 \alpha}{(2-\beta)^2-\alpha^2}=
\frac {4 (2-\beta)}{(2-\beta)^2-\alpha^2}\frac{\alpha}{2-\beta}
\]
as proposed.
\end{proof}

On the other hand, the proof Lemma \ref{lemma_w} also inspires the proof of Theorem \ref{CLT_m_crit}. Indeed, redoing the computations there shows that for $\alpha+\beta=2$ the quadratic term in the first component of $\chi_{N, \alpha, \beta}$ cancels. To this end we have to rescale $\tilde w_1$ to make the second term in the Taylor expansion $\log \cosh$ appear (as a matter of fact this is very similar, to what happens in the Curie-Weiss model at its critical temperature $\beta=1$.

\begin{proof}[Proof of Theorem \ref{CLT_m_crit}]
As motivated above we will now consider the vector $\hat w=(\hat w_1, \hat w_2)$ consisting of the components
\[
\hat w_1:= N^{1/4} w_1 \qquad \mbox{and } \quad \hat w_2:= \sqrt N w_2 = \tilde w_2.
\]
This time we will convolute the distribution of $\hat w$ under the Gibbs measure $\mu_{N,\alpha,\beta, S}$ with a two-dimensional Gaussian distribution $\mathcal{N}(0,\hat C)$, where
$
\hat C=  \begin{pmatrix}
\frac{1}{\sqrt N} & 0\\
0 & \frac{2}{\beta-\alpha}
\end{pmatrix}
$
(note that this is well defined since $\beta > \alpha$).
Computing the density of
$\hat \chi_{N,\alpha,\beta} :=
\mu_{N,\alpha,\beta}(\hat{w})^{-1} * {\mathcal N}(0, \hat C)$ as in the proof of Lemma \ref{lemma_w} we obtain
building on the fact that now $\alpha+\beta=2$
\begin{align*}
&  \hat \chi_{N,\alpha, \beta} (A) \\
 = & \hat K_1
\sum_{\sigma \in \{-1,1\}^N}
\int_{A} \exp \left( -\frac{1}{2}( \sqrt N(x- \hat{w_1})^2  +  \frac{\beta - \alpha}{2} (y-\hat {w_2})^2\right) \varrho_{N,\alpha, \beta}(\sigma) dx dy\\
 = & \hat K_2
\sum_{\sigma \in \{-1,1\}^N}
\int_{A}    \exp \left( -\frac{1}{2}( \sqrt N (x- \hat{w_1})^2  +  \frac{\beta - \alpha}{2} (y-\hat{w_2})^2\right)\\
&\hskip 3cm \times \exp\left(\frac{1}{4} (2 \sqrt N \hat{w_1}^2 + (\beta-\alpha) \tilde{w_2}^2)\right) dx dy\\
 = & \hat K_2
\sum_{\sigma \in \{-1,1\}^N}
\int_{A}    \exp \left( -\frac{\sqrt N}{2} x^2 -   \frac{1}{4} (  \beta - \alpha) y^2 \right)
\exp\left(\sqrt N x \hat{w_1} + \frac{1}{2} (\beta  - \alpha ) y \tilde{w_2}\right)dx dy\\
 = & \hat K_3
\int_{A}    \exp\left( -\frac{\sqrt N}{2}x^2 -   \frac{1}{4} (  \beta - \alpha) y^2 \right)\exp\left( \frac N2 \log \cosh \left( \frac{x}{N^{1/4}} + \frac{\beta- \alpha}{2\sqrt N} y\right)\right.\\
 & \qquad \qquad \left.+  \frac N2 \log \cosh \left( \frac{x}{N^{1/4}} - \frac{\beta-\alpha}{2\sqrt N} y\right)\right)dx dy
\end{align*}
with the normalizing constants $\hat K_1, \hat K_2$, and $\hat K_3$ chosen similarly to $K_1, K_2$, and $K_3$ in the proof of Lemma \ref{lemma_w}.
Now we expand the $\log \cosh$ to fourth order:
$
\log \cosh(z)= \frac{z^2} 2 - \frac 1 {12}  z^4 +\mathcal{O}(z^6).
$
We thus see that the $x^2$ terms in the exponent cancel and so do the $xy$-terms (fortunately). For fixed $x$ and $y$ only the $x^4$ is of vanishing order. The $y^2$ terms are treated as in the proof of Lemma \ref{lemma_w}.
We thus see that
\begin{align*}
-&\frac{\sqrt N}{2}x^2 -   \frac{1}{4} (  \beta - \alpha) y^2 + \frac N2 \log \cosh \left( \frac{x}{N^{1/4}} + \frac{\beta- \alpha}{2\sqrt N} y\right)+  \frac N2 \log \cosh \left( \frac{x}{N^{1/4}} - \frac{\beta-\alpha}{2\sqrt N} y\right)\\
=& -\frac{1}{12} x^4 -   \frac{1}{2} y^2\left( \frac{  \beta - \alpha}{2} -\left( \frac{  \beta -\alpha}{2}\right)^2 \right)+\mathcal{O}(N^{-\frac 12})
\end{align*}
with a $\mathcal{O}(N^{-\frac 12})$ term that depends on $x$ and $y$. To conclude the convergence of $\hat \chi_{N,\alpha, \beta} (A)$ we now proceed as in the proof of Lemma \ref{lemma_w}. Here we will only sketch the differences, because many steps are very similar. The exact steps are left to the reader.
The main differences to the above proof of Lemma \ref{lemma_w} is that the inner region is again $B(0,R)$, while the intermediate region now is the rectangle $[-  N^{\frac 14} r, N^{\frac 14} r] \times [-r\sqrt N,r\sqrt N]$ to take into account that the Taylor expansion gives another order for $x$ than for $y$. Correspondingly in the intermediate region the integrable function that dominates
\begin{align*}
&\exp\left[ -\frac{\sqrt N}{2}x^2 -   \frac{1}{4} (  \beta - \alpha) y^2 \right]
\\
& \quad \quad
\times \exp\left[ \frac N2 \log \cosh \left( \frac{x}{N^{1/4}} + \frac{\beta- \alpha}{2\sqrt N} y\right)   +  \frac N2 \log \cosh \left( \frac{x}{N^{1/4}} - \frac{\beta-\alpha}{2\sqrt N} y\right)\right]
\end{align*}
 is given by $\exp(-d_1 x^4-d_2y^2)$ for suitable constants $d_1, d_2>0$. With these changes we see that $\hat \chi_{N,\alpha, \beta} (A)$ converges to a 2-dimensional distribution with density proportional to
\[
\exp\left(-\frac{1}{12} x^4 -   \frac{1}{2} y^2\left( \frac{  \beta - \alpha}{2} -\left( \frac{  \beta -\alpha}{2}\right)^2 \right) \right)
\]
with respect to the 2-dimensional Lebesgue measure.  However, from here we see that $\hat w_1$ converges in distribution to a random variable with density proportional to $e^{-\frac{1}{12} x^4}$ since the Gaussian measure we convoluted the first coordinate of $\hat w$ with converges to 0 in probability. Moreover, the same computation as in Lemma \ref{lemma_w} shows that $\hat w_2= \tilde w_2=\frac{\sqrt{N}}{2}(m_1-m_2)$ converges to a normal distribution with mean $0$ and variance $\frac 2 {2-(\beta-\alpha)}$.

However, the latter convergence implies that $N^{\frac 14 }w_2$ converges to $0$ in probability. Thus $N^{\frac 14 }m_1= N^{\frac 14 } w_1+ N^{\frac 14 }w_2$ also converges in distribution to a random variable
with density proportional to $e^{-\frac{1}{12} x^4}$ (see e.g. \cite[Theorem 3.1]{billingsley}).
\end{proof}

\medskip
\noindent{\bf Acknowledgment}: We are grateful to an anonymous referee for many valuable remarks.


%
%

\end{document}